\documentclass{article}

    \PassOptionsToPackage{numbers, compress}{natbib}

\newif\ifpreprint
\preprinttrue
\preprintfalse

    \usepackage[final]{neurips_2022}

\usepackage[utf8]{inputenc} 
\usepackage[T1]{fontenc}    
\usepackage{hyperref}       
\usepackage{url}            
\usepackage{booktabs}       
\usepackage{amsfonts}       
\usepackage{nicefrac}       
\usepackage{microtype}      
\usepackage{xcolor}         

\usepackage{graphicx} 

\usepackage{amsmath}
\usepackage{amsthm}
\usepackage{amssymb}
\usepackage{physics}
\usepackage{enumerate}
\usepackage{color}
\usepackage{algorithm,algorithmic}

\newcommand{\fm}{f_{\bf m}}
\newcommand{\fn}{f_{\bf n}}
\newcommand{\Gm}{G_{\bf m}}
\newcommand{\Gn}{G_{\bf n}}
\newcommand{\hm}{h_{\bf m}}
\newcommand{\hn}{h_{\bf n}}

\newtheorem{theorem}{Theorem}
\newtheorem{lemma}{Lemma}
\newtheorem{definition}{Definition}
\newtheorem{proposition}{Proposition}
\newtheorem{corollary}{Corollary}

\newcommand{\R}{\mathbb{R}}
\newcommand{\T}{\mathrm{T}}

\title{Learning Deep Input-Output Stable Dynamics}

\author{%
  Ryosuke Kojima\thanks{Equal contribution.}\\
  Graduate School of Medicine\\
  Kyoto University\\
  Kyoto, 606-8501\\
  \texttt{kojima.ryosuke.8e@kyoto-u.ac.jp} \\
   \And
  Yuji Okamoto$^*$\\
  Graduate School of Medicine\\
  Kyoto University\\
  Kyoto, 606-8501\\
  \texttt{okamoto.yuji.2c@kyoto-u.ac.jp} \\
}

\begin{document}

\maketitle

\begin{abstract}
Learning stable dynamics from observed time-series data is an essential problem in robotics, physical modeling, and systems biology.
Many of these dynamics are represented as an inputs-output system to communicate with the external environment.
In this study, we focus on input-output stable systems, exhibiting robustness against unexpected stimuli and noise.
We propose a method to learn nonlinear systems guaranteeing the input-output stability.
Our proposed method utilizes the differentiable projection onto the space satisfying the Hamilton-Jacobi inequality to realize the input-output stability.
The problem of finding this projection can be formulated as a quadratic constraint quadratic programming problem, and we derive the particular solution analytically.
Also, we apply our method to a toy bistable model and the task of training a benchmark generated from a glucose-insulin simulator.
The results show that the nonlinear system with neural networks by our method achieves the input-output stability, unlike naive neural networks.
Our code is available at \url{https://github.com/clinfo/DeepIOStability}.

\end{abstract}


\section{Introduction}
Learning dynamics from time-series data has many applications such as industrial robot systems  \cite{swevers2007dynamic}, physical systems\cite{brunton2016discovering}, and biological systems \cite{gardner2003inferring,Roeder2019}.
Many of these real-world systems equipped with inputs and outputs to connect for each other, which are called {\it input-output systems} \cite{Khalil2002}.
For example, biological systems sustain life by obtaining energy from the external environment through their inputs.
Such real-world systems have various properties such as stability, controllability, and observability, which provide clues to analyze the complex systems.

Our purpose is to learn a complex system with ``desired properties'' from a dataset consisting of pairs of input and output signals.
To represent the target system, this paper considers the following nonlinear dynamics:
\begin{align}
\begin{aligned}
\dot{x} &= f(x) + G(x)u , \quad x(0) = x_0\\
y &= h(x).
\end{aligned}\label{Eq:main_system}
\end{align}
where the inner state $x$, the input $u$, and the output $y$ belong to a signal space that maps time interval $[0,\infty)$ to the Euclidean space.
We denote the dimension of $x$, $u$, and $y$ as $n$, $m$, and $l$, respectively.

Recently, with the development of deep learning, many methods to learn systems from time-series data using neural networks have been proposed \cite{krishnan2016structured,chen2018neural,zhong2019symplectic}.
By representing the maps $(f,G,h)$ in Eq.~(\ref{Eq:main_system}) as neural networks, complex systems can be modeled and trained from a given dataset.
However, guaranteeing that a trained system has the desired properties is challenging.

\begin{figure}[t]
    \centering
    \begin{tabular}{cc}
    \includegraphics[width=0.4\linewidth]{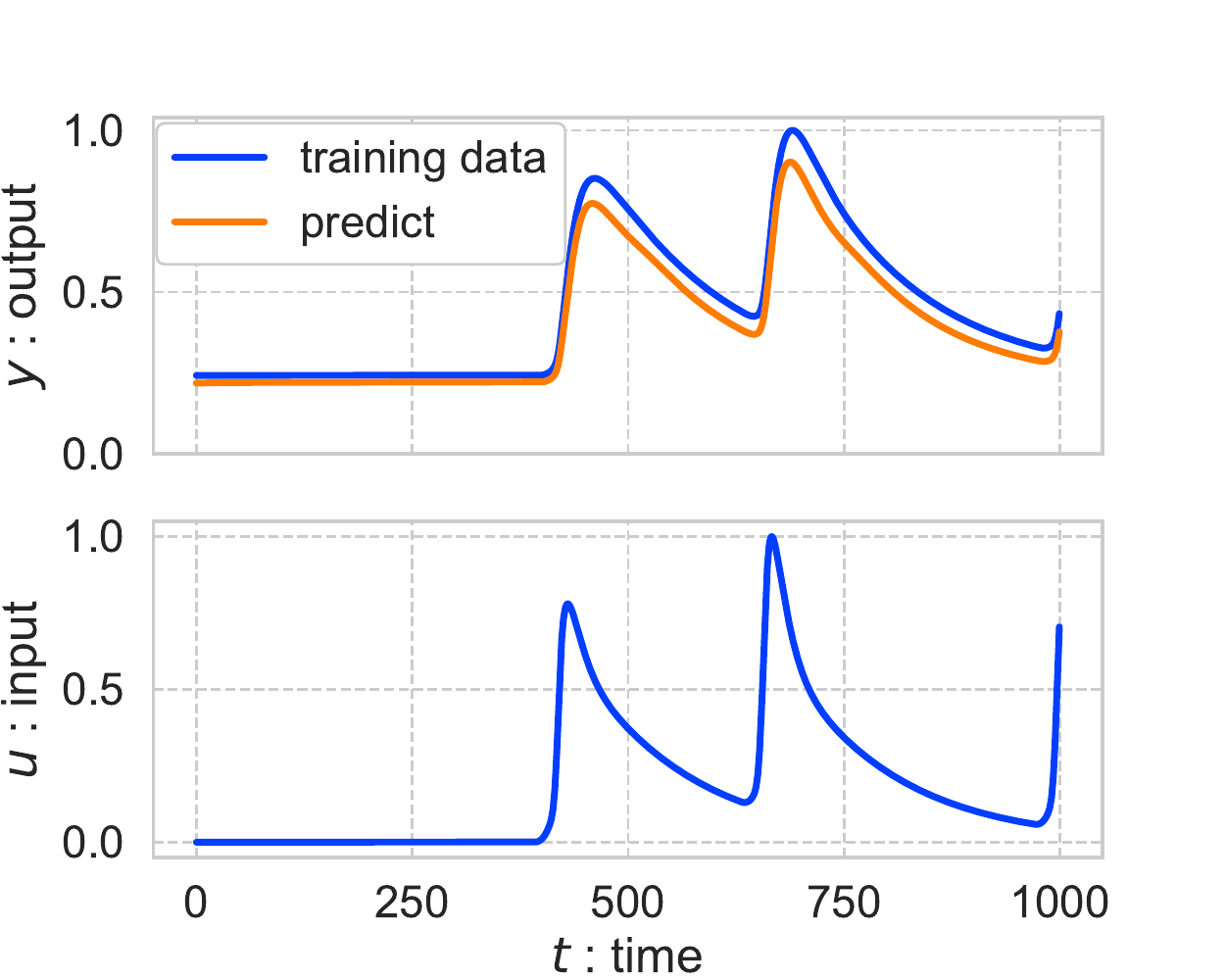}
    & \includegraphics[width=0.4\linewidth]{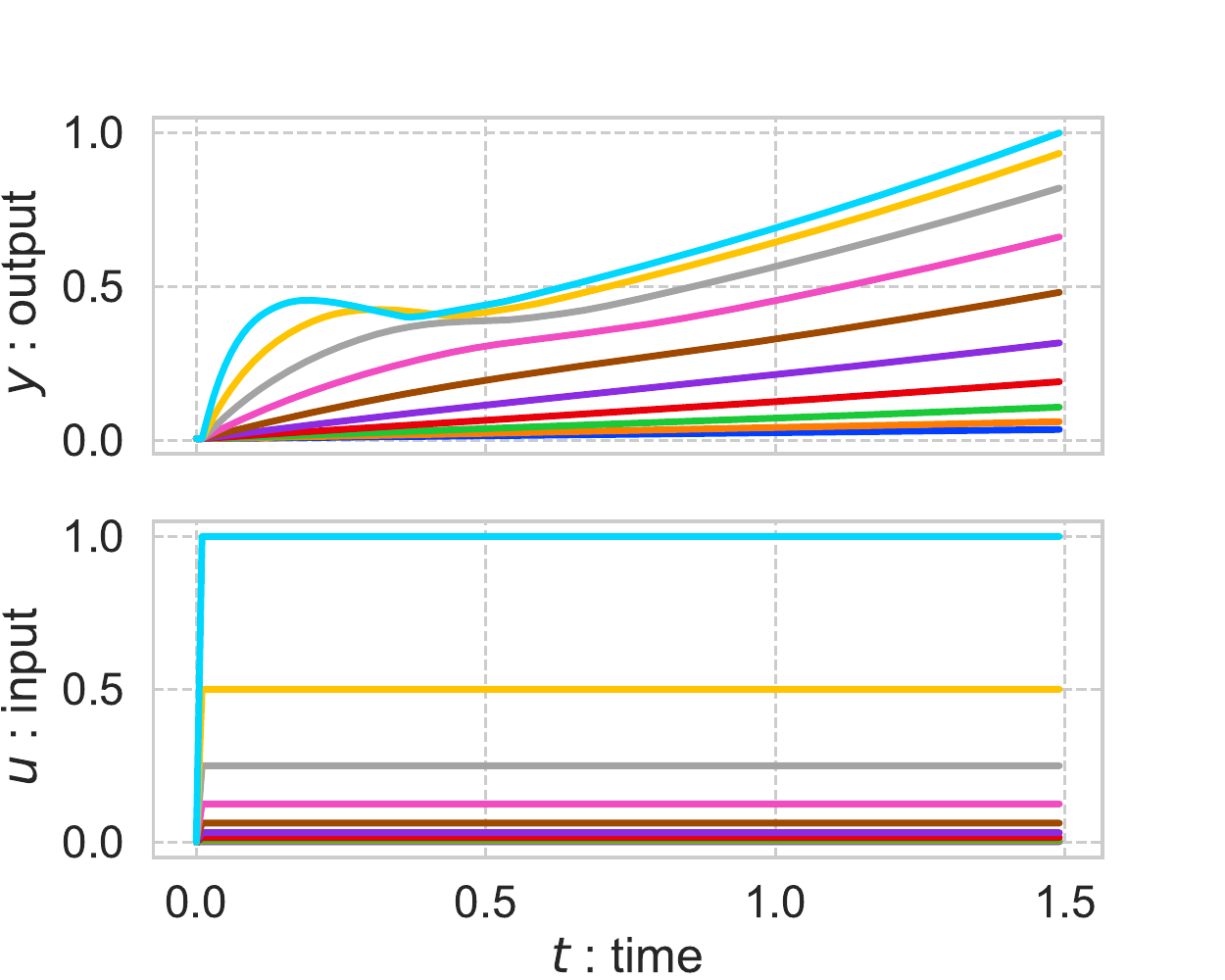}\\
        \text{(A)}&    \text{(B)} 
    \end{tabular}
    \caption{(A) The model prediction of neural networks  (B) the reaction of trained model. These results are min-max normalized.}
    \label{fig:introduction}
\end{figure}

A naively trained system fits the input signals contained in the training dataset, but does not always fit for new input signals.
For example, Figure~\ref{fig:introduction} shows our preliminary experiments
where we naively learned neural networks $(f,G,h)$ in Eq.~(\ref{Eq:main_system}) from input and output signals.
The trained neural networks provided small predictive errors for an input signal in the training dataset (Figure~\ref{fig:introduction} (A) ).
Contrastingly, the unbounded output signals were computed by the trained system with new step input signals  (Figure~\ref{fig:introduction} (B) ).
The reason can be expected that the magnitude (integral value) of this input signal was larger than that of the input signals in the training dataset.

The internal stability is known as one of the attractive properties that should be often satisfied in real-world dynamical systems.
The conventional methods to train internal stable systems consisting of neural networks adopt the Lyapunov-based approaches \cite{Manek2019,Lawrence2020,schlaginhaufen2021learning}.
These methods focus on the internal system, $\dot{x} = f(x)$ in the input-output system~(\ref{Eq:main_system}).
Thus, how to learn the entire  system~(\ref{Eq:main_system}) with the desired property related to the influence of the input signal is still challenging.

We propose a novel method to learn a dynamical system consisting of neural networks considering the input-output stability.
The notion of the input-output stability is often used together with the Hamilton-Jacobi inequality for controller design of the target system in the field of control theory.
The Hamilton-Jacobi inequality is one of the sufficient conditions for the input-output stability.
The feature of this condition is that the variable for an input signal $u$ does not appear in the expression, i.e., we do not need to evaluate the condition for unknown inputs \cite{2016Schaft}.
To the best of our knowledge, this is the first work that establishes a learning method for the dynamical systems consisting of neural networks using the Hamilton-Jacobi inequality.

The contributions of this paper are as follows:
\begin{itemize}
  \item This paper derives differentiable projections to the space satisfying the Hamilton-Jacobi inequality.
  \item This paper presents the learning method for the input-output system proven to always satisfy the Hamilton-Jacobi inequality.
  \item This paper also provides a loss function derived from the Hamilton-Jacobi inequality. By combining this loss function with the projection described above, efficient learning can be expected.
  \item This paper presents experiments using two types of benchmarks to evaluate our method.
\end{itemize}


\section{Background}
This section describes the $\mathcal{L}_2$ stability, a standard definition of the input-output stability, and the Hamilton-Jacobi inequality.

First, we define the $\mathcal{L}_2$ stability of the nonlinear dynamical system~(\ref{Eq:main_system}).
If the norm ratio of the output signal to the input signal is bounded, the system is $\mathcal{L}_2$ stable and this norm ratio is called the $\mathcal{L}_2$ gain.
The $\mathcal{L}_2$ norm on the input and output signal space is used for the definition of the $\mathcal{L}_2$ stability.
To deal with the $\mathcal{L}_2$ stability on the nonlinear system~(\ref{Eq:main_system}), we assume that the origin $x\equiv 0$ of the internal system $\dot{x} = f(x)$ is an asymptotically stable equilibrium point and $h(0) = 0$.
Since the translation of any asymptotically stable equilibrium points is possible, this assumption can be satisfied without loss of generality.

\begin{definition}
The $\mathcal{L}_2$ norm is defined as 
$\|x\|_{\mathcal{L}_2} :=  \sqrt{\int_{0}^\infty \|x(t)\|^2 \dd t}$.
If there exists $\gamma \geq0$ and a function $\beta$ on a domain $D\subset \R^n$ such that
\begin{align}
 \|y\|_{\mathcal{L}_2}\leq \gamma \|u\|_{\mathcal{L}_2}+ \beta(x_0), \label{Eq:L2_Stability}
\end{align}
then the system (\ref{Eq:main_system}) is $\mathcal{L}_2$ stable, where the function $\beta(\cdot)$ is non-negative and $\beta(0)=0$.
Furthermore, the minimum $\gamma$ satisfying (\ref{Eq:L2_Stability}) is called the $\mathcal{L}_2$ gain of the nonlinear system~(\ref{Eq:main_system}).
\end{definition}

Next, we describe a sufficient condition of the $\mathcal{L}_2$ stability using the Lyapunov function $V:D \rightarrow \R$, where the function $V$ is positive semi-definite, i.e., $V(x)\geq 0$ and $V(0) = 0$.
The Hamilton-Jacobi inequality is an input-independent sufficient condition.

\begin{proposition}[{\cite[Theorem 5.5]{Khalil2002}}]~\label{Prop:IO_Stability}
Let $f$ be locally Lipschitz and let $G$ and $h$ be continuous.
If there exist a constant $\gamma > 0$ and a continuously differentiable positive semi-definite function $V:D\subset \R^n\rightarrow \R$ such that 
\begin{align}
\begin{aligned}
&\nabla V^\T(x) f(x) + \frac{1}{2\gamma^2}\|G^\T(x)\nabla V(x)\|^2 + \frac{1}{2} \|h(x)\|^2\leq 0,\quad  \forall x\in D\setminus \{ 0\},
\end{aligned}
\label{Eq:HJI}
\end{align}
then the system (\ref{Eq:main_system}) is $\mathcal{L}_2$ stable and the $\mathcal{L}_2$ gain is less than or equal to $\gamma$.
This condition is called the Hamilton-Jacobi inequality.
\end{proposition}

The above proposition can be generalized to allow the more complicated situations such as limit-cycle and bistable cases, where the domain $D$ contains multiple asymptotically stable equilibrium points.
The equilibrium point assumed in this proposition can be replaced with positive invariant sets by extending the $\mathcal{L}_2$ norm \cite{haddad2008}.
Furthermore, by mixing multiple Lyapunov functions, this proposition can be generalized around multiple isolated equilibrium points.

\section{Method}
The goal of this paper is to learn the $\mathcal{L}_2$ stable system represented by using neural networks $(f, G, h)$.
The Hamiltonian-Jacobi inequality, which implies the $\mathcal{L}_2$ stability, is expressed by $(f, G, h)$.
We present a method to project $(f,G,h)$ onto the space where the Hamilton-Jacobi inequality holds.

\subsection{Modification of nonlinear systems}
Supposing $\fn :\R^n\rightarrow \R^n$, $\Gn:\R^n\rightarrow \R^{n\times m}$, and $\hn:\R^n\rightarrow \R^{l}$, a triplet map $(\fn, \Gn, \hn)$ denote as {\it nominal} dynamics.
Introducing a distance on the triplet maps, the nearest triplet satisfying the Hamilton-Jacobi inequality from the nominal dynamics $(\fn,\Gn,\hn)$ is called modified dynamics $(\fm,\Gm,\hm)$.
The purpose of this section is to describe the modified dynamics $(\fm,\Gm,\hm)$ associated with the nominal dynamics $(\fn,\Gn,\hn)$ by analytically deriving a projection onto the space satisfying the Hamilton-Jacobi inequality.

The problem of finding the modified dynamics $(\fm,\Gm,\hm)$ is written as a quadratic constraint quadratic programming (QCQP) problem for the nominal dynamics $(\fn,\Gn,\hn)$.
Since there is generally no analytical solution for QCQP problems, we aim to find the particular solution by adjusting the distance on the triplets.

To prepare for the following theorem, we define the ramp and the clamp functions as
\begin{align*}
    \mathrm{R}(x) \triangleq 
    \begin{cases}
    0, &  x \leq 0\\
    x, & x>0
    \end{cases},\quad 
    \mathrm{C}(x;a,b)\triangleq
    \begin{cases}
    a, & x \leq a\\
    x, & a < x \leq b\\
    b, & x > b
    \end{cases},
\end{align*}
and define the Hamilton-Jacobi function as
\begin{align}
    \mathrm{HJ}(f,G,h) \triangleq \nabla V^\T f + \frac{1}{2\gamma^2}\| G^\T \nabla V  \|^2 + \frac{1}{2}\|h\|^2, \label{Eq:HJ_function}
\end{align}
where $V$ is a given positive definite function.
A way to design $V$ is to determine the desired positive invariant set and design the increasing function around this set. For example, if  the target system has a unique stable point, $x=0$ can be regarded as the positive invariant set and the increasing  function $V$ can be designed as $V(x) = \tfrac{1}{2}x^2$.

\begin{theorem}\label{Thm:optimization}
Consider that the following optimal problem:
\begin{subequations}
\begin{align}
    &\underset{\fm, \Gm, \hm}{\text{\rm \bf minimize}}\quad\frac{(1-k)}{\|\nabla V\|}\|\fm-\fn\| + \frac{k}{2\gamma^2}\|\Gm-\Gn\|^2 + \frac{k}{\|\nabla V\|^2}\|\hm - \hn\|^2 \label{Eq:thm_optimization_minimize}\\
    &\text{\rm \bf subject to}\quad \mathrm{HJ} (\fm,\Gm,\hm) \leq 0,\label{Eq:thm_optimization_subject}
\end{align}\label{Eq:thm_optimization}
\end{subequations}
where $k \in [0,1]$.
The solution of (\ref{Eq:thm_optimization}) is given by
\begin{align*}
    \fm &= \fn -\frac{1}{\|\nabla V\|^2}\mathrm{R}\left(V_{f} + k^2 V_{Gh}\right) \nabla V,\\
    \Gm &= \Gn -\left(1 - \sqrt{\mathrm{C}\left( - \tfrac{V_{f}}{V_{Gh}};k^2,1\right)} \right)P_{V} \Gn,\\
    \hm &= \sqrt{\mathrm{C}\left(  - \tfrac{V_{f}}{V_{Gh}};k^2,1\right)}\hn,
\end{align*}
where
\begin{align*}
V_{f}\triangleq\nabla V^{\T} \fn,
\quad V_{Gh}\triangleq\frac{1}{2\gamma^2} \|\Gn^\T \nabla V\|^2 + \frac{1}{2}\|\hn\|^2,
\quad P_{V}\triangleq \frac{\nabla V \nabla V^\T}{\|\nabla V\|^2}.
\end{align*}
\end{theorem}
\begin{proof}
See Appendix~\ref{APP1}.
\end{proof}

The objective function of (\ref{Eq:thm_optimization_minimize}) is a new distance between the nominal dynamics $(\fn, \Gn, \hn)$ and the modified dynamics $(\fm,\Gm,\hm)$.
This new distance allows the derivation of analytical solutions by combining three distances of $f$, $G$, and $h$.

Focusing on the objective function~(\ref{Eq:thm_optimization_minimize}), the constant $k$ represents the ratio of the distance scale of $f$ to $G$ and $h$.
When $k = 0$, the result of this problem~(\ref{Eq:thm_optimization}) are consistent with the projection of the conventional method that guarantee internal stability \cite{Manek2019}.
As the constant $k$ converges $1$, the modification method of Theorem~\ref{Thm:optimization} approaches $\Gm$ and $\hm$ to $\Gn$ and $\hn$, respectively.
In this case, the objective function~(\ref{Eq:thm_optimization_minimize}) becomes the distance between $\fm$ to $\fn$.
Therefore, the following corollary is satisfied.
\begin{corollary}\label{Coro:optimization_fix_Gh}
The solution of 
\begin{subequations}
\begin{align}
    &\underset{\fm}{\text{\rm \bf minimize}}\quad \|\fm-\fn\| \label{Eq:coro1_optimization_minimize} \\
    &\text{\rm \bf subject to}\quad \mathrm{HJ} (\fm,\Gn,\hn) \leq 0,\label{Eq:coro1_optimization_subject}
\end{align}\label{Eq:coro_optimization}
\end{subequations}
is given by
\begin{align*}
    \fm &= \fn -  \frac{1}{\|\nabla V\|^2}\mathrm{R}\left({\rm HJ}(\fn,\Gn,\hn) \right) \nabla V.
\end{align*}
\end{corollary}
\begin{proof}
This solution is easily derived from Theorem~\ref{Thm:optimization}.
\end{proof}

Corollary~\ref{Coro:optimization_fix_Gh} derives a solution of a linear programming problem rather than QCQP problems.
Replacing the Hamilton-Jacobi function $\mathrm{HJ}(\fm,\Gn,\hn)$ with the time derivative of a positive definite function $\nabla V^\T f$, this corollary matches the result of the conventional study \cite{Manek2019}.

When the map $\hm$ is fixed as $\hn$, a similar solution as Theorem~\ref{Thm:optimization} is derived.
Although Corollary~\ref{Coro:optimization_fix_Gh} is proved by changing $k$, the modified dynamics with the fixed $\hn$ are not derived.
We reprove this modified dynamics in a similar way to Theorem~\ref{Thm:optimization}.
\begin{corollary}\label{Coro:optimization_fix_h}
Consider the following problem:
\begin{subequations}
\begin{align}
    &\underset{\fm, \Gm}{\text{\rm \bf minimize}}\quad \frac{(1-k)}{\|\nabla V\|}\|\fm-\fn\| + \frac{k}{2 \gamma^2} \|\Gm-\Gn\|^2 \label{Eq:coro2_optimization_minimize} \\
    &\text{\rm \bf subject to}\quad \mathrm{HJ} (\fm,\Gm,\hn) \leq 0, \label{Eq:coro2_optimization_subject}
\end{align}\label{Eq:coro2_optimization}
\end{subequations}
where $k\in [0,1]$.
The solution of (\ref{Eq:coro2_optimization}) is given by
\begin{align}
    \fm = \fn -  \frac{1}{\|\beta\|^2}\mathrm{R}\left(V_{fh} + k^2V_{G}\right) \nabla V,\quad
    \Gm= \Gn -\left(1 - \sqrt{\mathrm{C}\Big( - \tfrac{V_{fh}}{V_{G}};k^2,1\Big)} \right)P_{V} \Gn,
\label{Eq:coro_optimization_solution}
\end{align}
where
\begin{align*}
V_{fh}\triangleq\nabla V^{\T} \fn + \frac{1}{2}\|\hn\|^2,
\quad V_{G}\triangleq\frac{1}{2\gamma^2} \|\Gn^\T \nabla V\|^2.
\end{align*}
\end{corollary}
\begin{proof}
See Appendix~\ref{APP1}.
\end{proof}

\subsection{Loss function}

\begin{figure}[t]
    \centering
     \includegraphics[width=1.0\linewidth]{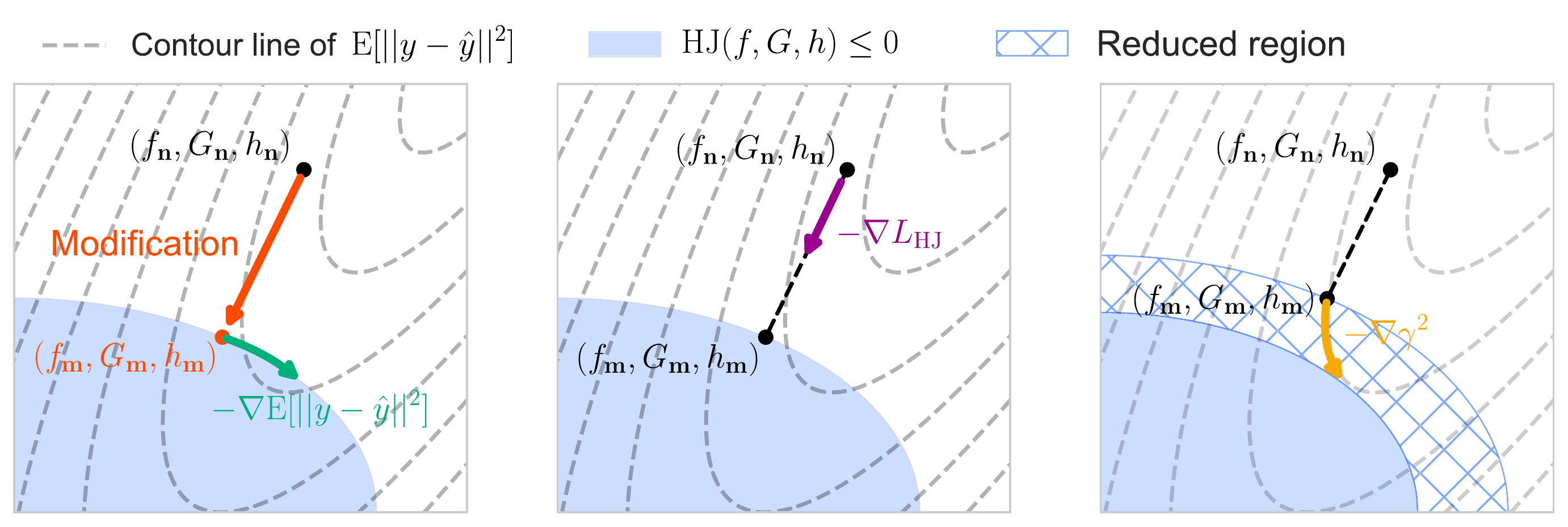}
    \text{(A)\hspace{27ex} (B)\hspace{27ex}(C)}
    \caption{Sketches of our method :
    (A) minimizing the prediction error (the first term of the loss function~(\ref{Eq:HJ_Loss})) in the blue region,
    (B) moving the nominal dynamics to the blue region (the second term),
    and (C) reducing the blue region while keeping the same level of the prediction error (the last term).
    }
    \label{fig:method_sketch}
\end{figure}

We represent $(\fn,\Gn,\hn)$ as neural networks and denote $(\fm,\Gm,\hm)$ as the $\mathcal{L}_2$ stable dynamics modified by  Theorem~\ref{Thm:optimization}, Corollary~\ref{Coro:optimization_fix_Gh}, or  \ref{Coro:optimization_fix_h}.
Note that the modification depends on the candidate of $\mathcal{L}_2$ gain $\gamma$.
Figure~\ref{fig:method_sketch}~(A) shows the sketch of this modification, where the blue region satisfies the Hamilton-Jacobi inequality.

Since the nonlinear system of the modified dynamics $(\fm,\Gm,\hm)$ is represented as ordinary differential equations (ODEs) consisting of the differentiable functions $\fn$,$\Gn$, and $\hn$, the techniques of training neural ODEs can be applied \cite{chen2019review}.
Once a loss function is designed, the parameters of neural networks in the modified system can be learned from given data.

\begin{align}
\mathrm{Loss}& = \mathrm{E}_{(x_0,u,\hat{y}) \in \cal{D}}[||y - \hat{y}||_{\mathcal{L}_2}^2] +\lambda L_{\rm HJ} + \alpha \gamma^2, \label{Eq:HJ_Loss}
\end{align}
where $\lambda$ and $\alpha$ are non-negative coefficients.

The first term shows the prediction error of the output signal $y$ ( Figure~\ref{fig:method_sketch}~(A)).
A dataset $\cal{D}$ consists of tuples $(x_0,u,y)$ where the initial value $x_0$, the input signal $u$, and the output signal $y$. 
The predicted output $\hat{y}$ is calculated from $x_0$, $u$, and the modified dynamics $(\fm,\Gm,\hm)$.

The second term aims to improve the nominal dynamics $(\fn,\Gn,\hn)$ closer to the modified dynamics $(\fm,\Gm,\hm)$ and is defined as
\begin{align*}
L_{\mathrm{HJ}}& =  \mathrm{E}_{x }[\mathrm{R}(\mathrm{HJ}(\fn,\Gn,\hn)(x) + \varepsilon)],
\end{align*}
where $\varepsilon$ is a positive constant ( Figure~\ref{fig:method_sketch}~(B)).
Since this term is a form of the hinge loss, $\varepsilon$ represents the magnitude of the penalties for the Hamilton-Jacobi inequality.
To evaluate this inequality for any $x\in D$, we introduce a distribution of $x$ over the domain $D$.
In our experiments, this distribution is decided as a Gaussian distribution ${\cal N}(\mu,\sigma^2)$ where the mean $\mu$ is placed at the asymptotically stable point and the variance $\sigma^2$ is an experimental parameter.
Without this second term of the loss function, there are degrees of freedom in the nominal dynamics, i.e., multiple nominal dynamics give the same loss by the projection, which negatively affects the parameter learning.

The modifying parameter $\gamma$ can be manually designed for the application or automatically trained from data by introducing the last term.
This training explores smaller $\gamma$ while keeping the same level of the prediction error (Figure~\ref{fig:method_sketch}~(C)).

\section{Related work}
Estimating parameters of a given system is traditionally studied as system identification.
In the field of system identification, much research on the identification of linear systems have been done,
where the maps $(f,G,h)$ in Eq.~(\ref{Eq:main_system}) assumes to be linear \citep{2005Katayama}.
Linear state-space models include identification methods by impulse response like the Eigensystem Realization Algorithm (ERA)\cite{juang1985eigensystem}, by the state-space representation like Multivariable Output Error State sPace (MOESP)\cite{verhaegen1992subspace} and ORThogonal decomposition method (ORT)\cite{katayama2005simple}.
System identification methods for non-state-space models, unlike our target system, contain Dynamic Mode Decomposition with control (DMDc)\cite{proctor2016dynamic} and its nonlinear version, Sparse Identification of Nonlinear DYnamics with control (SINDYc)\cite{brunton2016sparse}.
Also, the nonlinear models such as the nonlinear ARX model \cite{Garulli2012} and the Hammerstein-Wiener model \cite{Wills2013identification} have been developed and often trained by error minimization.
For example, the system identification method using a piece-wise ARX model allows more complicated functions by relaxing the assumption of linearity \cite{Garulli2012}.

These traditional methods often concern gray-box systems, where the maps $(f,G,h)$ in Eq.~(\ref{Eq:main_system}) are partially known \cite{ljung1998system}.
This paper deals with a case of black-box systems, where $f$, $G$, and $h$ in the system~(\ref{Eq:main_system}) are represented by neural networks.
Our method can be used regardless of whether all $f$, $G$, and $h$ functions are parameterized using neural networks.
So, our method can be easily applied to application-dependent gray-box systems when the functions $f$, $G$, and $h$ are differentiable with respect to the parameters.

Because the input-output system~(\ref{Eq:main_system}) can be regarded as a differential equation, our study is closely related to the method of combining neural networks and ODEs \cite{chen2018neural,chen2019review}.
These techniques have been improved in recent years, including discretization errors and computational complexity.
Although we used an Euler method  for simplicity,
we can expect that learning efficiency would be further improved by using these techniques.

The internal stability is a fundamental property in ODEs, and learning a system with this property using neural networks plays an important role in the identification of real-world systems \cite{braun1983differential}.
In particular, the first method to guarantee the internal stability of the trained system has been proposed in \cite{Manek2019}.
Furthermore, another method \cite{takeishi2021learning} extends this method to apply positive invariant sets, e.g. limit cycles and line attractors.
Encouraged by these methods based on the Lyapunov function, our method further generalizes these methods using the Hamilton-Jacobi inequality to guarantee the input-output stability.

In this paper, the Lyapunov function $V$ is considered to be given, but this function can be learned from data.
Since a pair of dynamics (\ref{Eq:main_system}) and $V$ has redundant degrees of freedom, additional assumptions are required to determine $V$ uniquely. In \cite{Manek2019}, it is realized by limiting the dynamics to the internal system and restricting $V$ to a convex function \cite{amos2017input}.

Lyapunov functions are also used to design controllers, where the whole system should satisfy the stability condition.
A method for learning such a controller using neural nets has been proposed \cite{Chang2019dynamics}.
This method deals with optimization problems over the space that satisfies the stability condition, which is similar to our method.
Whereas we solve the QCQP problem to derive the projection onto the space, this method uses a Satisfiability Modulo Theories (SMT) solver to satisfy this condition.
The method has also been extended to apply unknown systems \cite{zinage2022neural}.

Although this paper deals with deterministic systems, neural networks for stochastic dynamics with the variational inference is well studied  \cite{krishnan2016structured,gedon2021deepssm,chung2015recurrent,bayer2014learning}.
Guaranteeing the internal stability of trained stochastic dynamics is important for noise filtering and robust controller design \cite{Lawrence2020}.

\section{Experiments}\label{sec:result}

We conduct two experiments to evaluate our proposed method.
The first experiment uses a benchmark dataset generated from a nonlinear model with multiple asymptotically equilibrium points.
In the next experiment, we applied our method to a biological system using a simulator of the glucose-insulin system.

\subsection{Experimental setting}
Before describing the results of the experiments, this section explains the evaluation metrics and our experimental setting.

For evaluation metrics, we define the root mean square error (RMSE) and the average $\mathcal{L}_2$ gain (GainIO) for the given input and output signals as follows:
\begin{align*}
\text{RMSE} \triangleq \sqrt{ \frac{1}{N}\sum_{i=1}^N \| y_i-\hat{y}_i \|^2_{\mathcal{L}_2}},\quad \text{GainIO} \triangleq \frac{1}{N}\sum^N_{i=1} \frac{\|\hat{y}_i\|_{\mathcal{L}_2} }{ \|u_i\|_{\mathcal{L}_2}},
\end{align*}
where $N$ is the number of signals in the dataset.
$u_i(\cdot)$ and $y_i(\cdot)$ are the input and output signal at the $i$-th index, respectively.
The prediction signal $\hat{y}_i(\cdot)$ is computed from $u_i(\cdot)$, the trained dynamics, and the initial state.
Note that the integral contained in the $\mathcal{L}_2$ norm is approximated by a finite summation.
The RMSE  is a metric of the prediction errors related to the output signal, and the GainIO is a metric of the property of the $\mathcal{L}_2$ stability. 
Whether the target system satisfies or does not satisfy the $\mathcal{L}_2$ stability, the GainIO with a given finite-size dataset can be calculated.
The GainIO error is defined by the absolute error between the GainIO of the test dataset and that of the prediction.

In our experiments, 90$\%$ of the dataset is used for training and the remaining 10$\%$ is used for testing.
We retry five times for all experiments and show the mean and standard deviations of the metrics.

For simplicity in our experiments, the sampling step $\Delta t$ for the output $y$ is set as constant and the Euler method is used to solve ODEs.
$x_0$ is put at an asymptotically stable point for each benchmark and is known.
In this experiment, to prevent the state from diverging during learning of dynamics, the clipping operation is used so that the absolute values of the states are less than ten.

For comparative methods, we use vanilla neural networks, ARX, ORT \cite{katayama2005simple}, MOESP \cite{verhaegen1992subspace}, and piece-wise ARX\cite{Garulli2012}.
In the method of vanilla neural networks, the maps $(f,G,h)$ in the nonlinear system~(\ref{Eq:main_system}) is represented by using three neural networks, i.e., this method is consistent with a method used in Figure~\ref{fig:introduction}.
To determine the hyperparameters of comparative methods except for neural networks, the grid search is used.
Note that these comparative methods only consider the prediction errors.

For training each method with neural networks, an NVIDIA Tesla T4 GPU was used.
Our experiments are totally run on 20 GPUs over about three days.

\subsection{Bistable model benchmark}

\begin{figure}[t]
    \centering
    \includegraphics[width=0.8\linewidth]{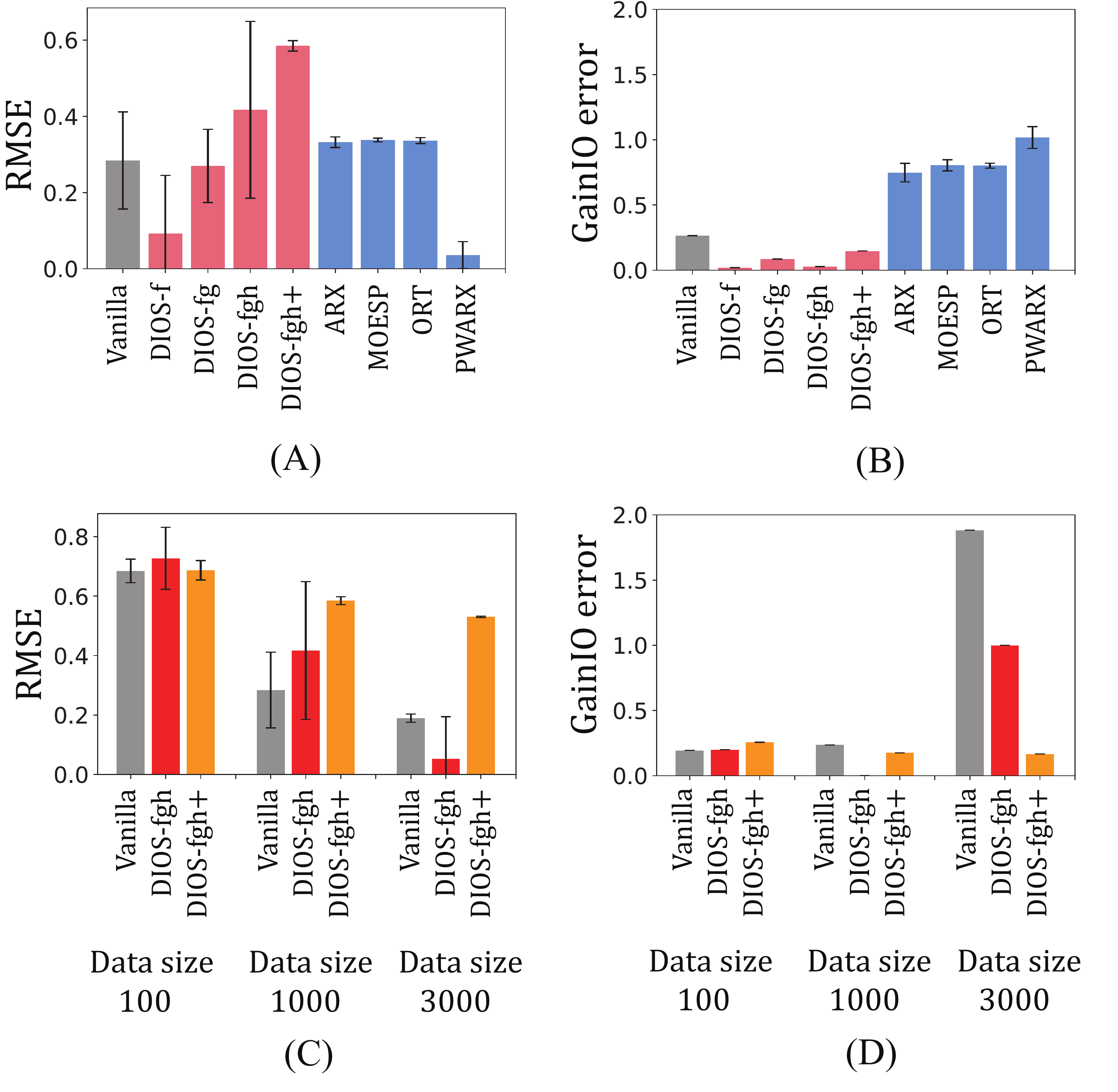}
    \caption{Results of the bistable model benchmark.
    The upper part shows (A) the RMSE and (B) the GainIO error of the vanilla neural networks (gray), our proposed methods (red), and the conventional methods (blue). The lower part shows (C) the RMSE and (D) the GainIO error for the different sizes of the datasets.}
    \label{fig:TBS_results}
\end{figure}

The first experiment is carried out using a bistable model, which is known as a bounded system with multiple asymptotically equilibrium points.
This bistable model is defined as
\begin{align*}
\dot{x} =  x(1-x^2) + u,\quad x(0) = -1,\quad y = x.
\end{align*}
The internal system of this model has two asymptotically stable equilibrium points $x \equiv 1, -1$.

We generate 1000 input and output signals for this experiment.
To construct this dataset, we prepare input signals using positive and negative pulse wave signals whose pulse width is changed at random.
The input and output signals on the period $[0,10]$ are sampled with an interval $\Delta t = 0.1$.
In this benchmark, we set the number of dimensions of the internal system as one and use a fixed function $V(x)=\min((x-1)^2,(x+1)^2)$, a mixture of the two positive definite functions. 

In the result of our experiments, we name the proposed methods modified by Theorem~\ref{Thm:optimization}, Corollary~\ref{Coro:optimization_fix_Gh}, and  \ref{Coro:optimization_fix_h} as DIOS-fgh, DIOS-f, and DIOS-fg, respectively.
In these methods, the parameters of our loss function are set as  $\lambda=0$ and $\alpha=0.01$.
Also, DIOS-fgh+ uses $\lambda=0.01$ and $\alpha=0.01$ under the same conditions as DIOS-fgh. 
For this example (1000 input signal), it took about 1 hour using 1 GPU to learn one model training.

The results of the RMSE and the GainIO error in this experiment are shown in Figure~\ref{fig:TBS_results}.
Figures~\ref{fig:TBS_results} (A) and (B) demonstrate that a piece-wise ARX model (PWARX) gives very low RMSE but its GainIO error is high.
Our proposed methods achieve a small GainIO error while keeping the RMSE.
Note that linear models such as MOESP and ORT only approach one point although the bistable model has two asymptotically stable equilibrium points.
Figures~\ref{fig:TBS_results} (C) and (D) display the effect of dataset sizes.
When the dataset size was varied to 100, 1000, and 3000, the result of the larger dataset provided the smaller RMSE for all methods.
The vanilla neural networks do not consider the $\mathcal{L}_2$ gain, so the GainIO error was large in the case of the low RMSE.

\begin{figure}[t]
    \centering
\includegraphics[width=0.8\linewidth]{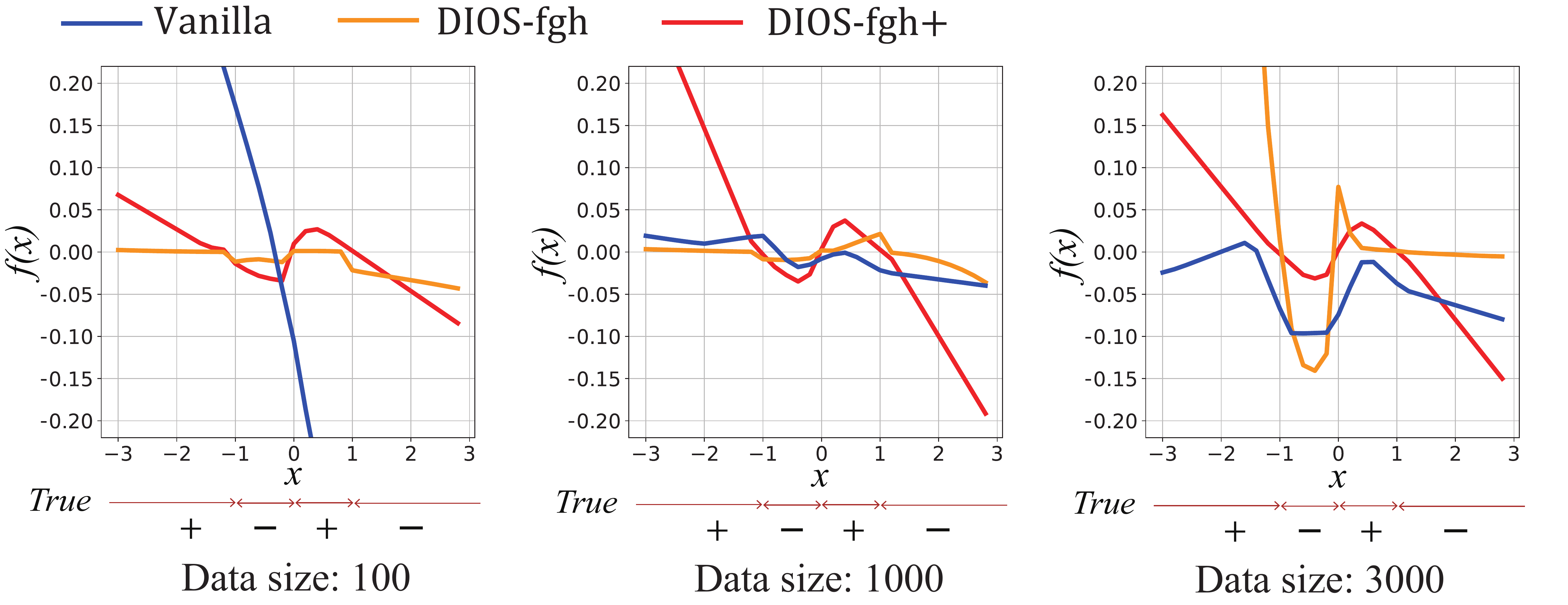}
\caption{The sketch $x$-$f(x)$ displaying the internal dynamics trained from the bistable datasets with the different sizes.
The sign of the bistable model is shown at the bottom of each figure.
}
    \label{fig:TBS_bistable_fx}
\end{figure}

Figure~\ref{fig:TBS_bistable_fx} shows the relationship between $x$ and $f(x)$ in the trained system in (\ref{Eq:main_system})  to compare the vanilla neural networks, DIOS-fgh and DIOS-fgh+.
DIOS-fgh and DIOS-fgh+ successfully find two stable points, i.e., the trained function $f(x)$ had roots of two stable points $x = \pm1$ and an unstable point $x = 0$.
Especially, these stable points were robustly estimated in the trained system using our presented loss function (DIOS-fgh+) even when the dataset size was small.
The vanilla neural networks failed to obtain the two stable points in all cases. 

\subsection{Glucose-insulin benchmark}
This section addresses an example of the identification of biological systems.
We consider the task of learning the glucose-insulin system using a simulator \cite{DEGAETANO200041} to construct responses for various inputs and evaluate the robustness of the proposed method for unexpected inputs.
This simulator outputs the concentrations of plasma glucose $y_1$ and insulin $y_2$ for the appearance of plasma glucose per minute $u$.
To determine the realistic input $u$, we adopt another model, an oral glucose absorption model \cite{DallaMan2007}.

Using this simulator, 1000 input and output signals are synthesized for this experiment.
The input and output signals are sampled with a sampling interval  $\Delta t = 1$ and 1000 steps for each sequence.
In this benchmark, we set the number of dimensions of the internal system as six, fix a positive definite function $V(x)=x^2$, and use our loss function with $\lambda=0.001$ and $\alpha=0.001$.
Training one model for this examples (1000 input signal) tooks about 7.5hours using 1GPU.
The hyperparameters including the number of layers in the neural networks, the learning rate, optimizer, and the weighted decay are determined using the tree-structured Parzen estimator (TPE) implemented in Optuna  \cite{optuna_2019}.


\begin{figure}[t]
    \centering
    \includegraphics[width=0.90\linewidth]{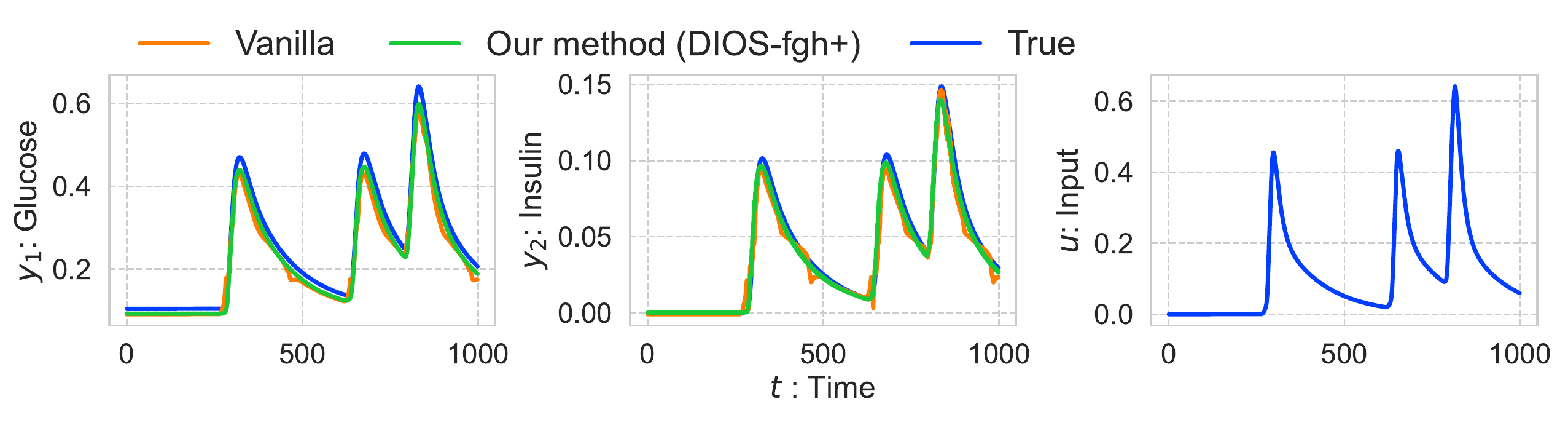}
    \caption{The input and output signals of the glucose-insulin simulator and the predicted output.}
    \label{fig:glcose_insulin_model_reaction}
\end{figure}

\begin{figure}[t]
    \centering
    \includegraphics[width=0.9\linewidth]{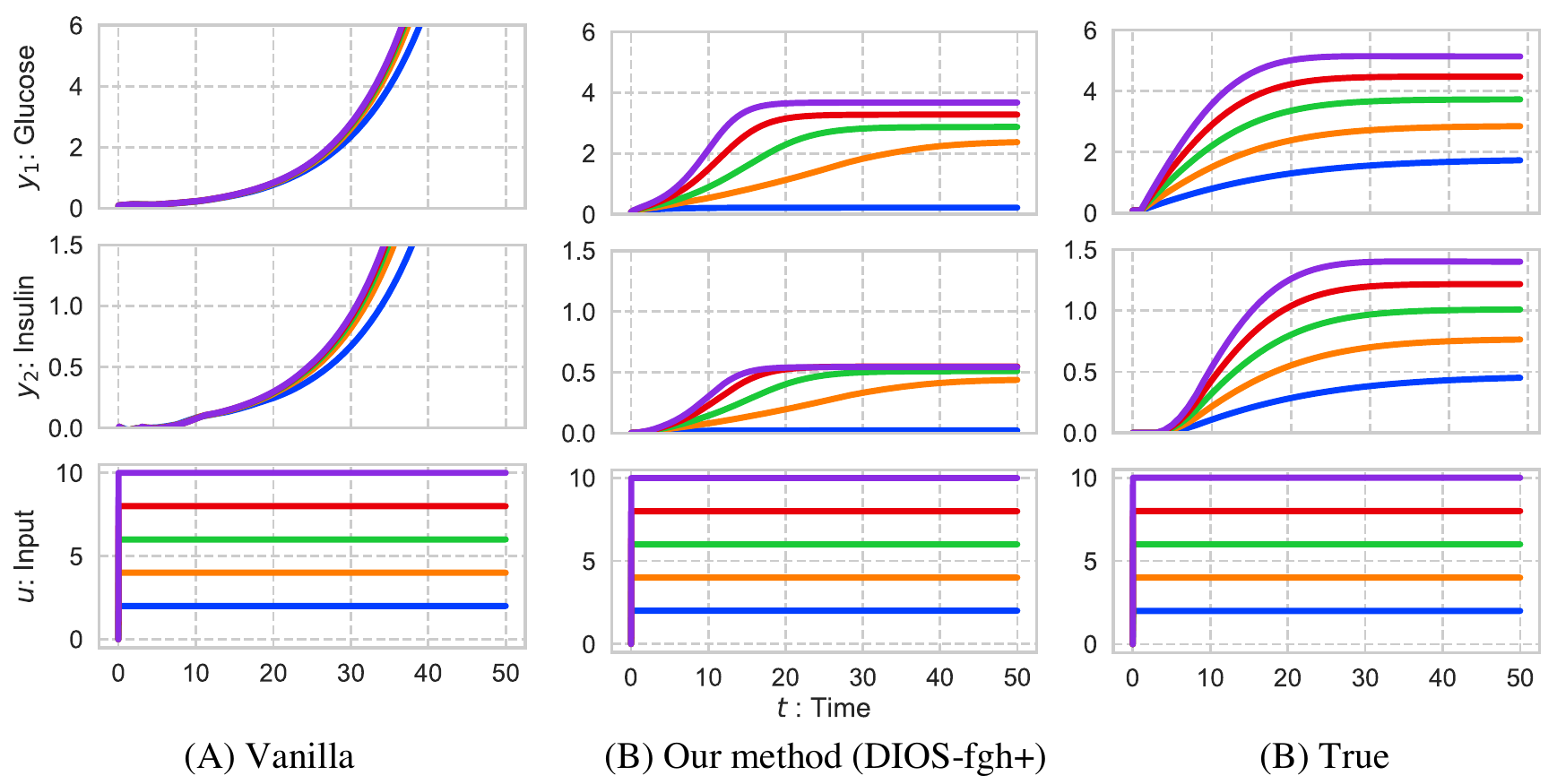}
    \caption{The step reaction of the trained systems. The color of each line indicates the magnitude of the step input.}
    \label{fig:glcose_insulin_model_step_reaction}
\end{figure}

The RMSE of vanilla and our proposed method are 0.0103 and 0.0050, respectively.
So, from the perspective of RMSE, these methods achieved almost the same performance.
Figure~\ref{fig:glcose_insulin_model_reaction} shows input and output signals in the test dataset and the predicted output by vanilla neural networks and our method (DIOS-fgh+).
The $\mathcal{L}_2$ stability of the system using the vanilla neural networks is not guaranteed.
Since the proposed method guarantees the $\mathcal{L}_2$ stability, the output signals of  DIOS-fgh+ are bounded even if the input signals are unexpectedly large.

To demonstrate this, we conducted an additional experiment using the trained system.
Figure~\ref{fig:glcose_insulin_model_step_reaction} shows the transition of output behavior caused by the magnitude of the input signal changing from 2 to 10.
Note that the maximum magnitude in the training dataset is one.
In this experiment, $\Delta t$ was changed to $0.01$ and the clipping operation was removed to deal with the large values of the state.

This result shows the output of the vanilla neural networks quickly diverged with an unexpectedly large input.
Contrastingly, the output behavior of our proposed method is always bounded.
Therefore, we actually confirmed that our proposed method satisfies the $\mathcal{L}_2$ stability.


\section{Conclusion}
This paper proposed a learning method for nonlinear dynamical system guaranteeing the $\mathcal{L}_2$ stability.
By theoretically deriving the projection of a triplet $(f, G, h)$ to the space satisfying the Hamilton-Jacobi inequality, our proposed method realized the $\mathcal{L}_2$ stability of trained systems.
Also, we introduced a loss function to empirically achieve a smaller $\mathcal{L}_2$ gain while reducing prediction errors.
We conducted two experiments to learn dynamical systems consisting of neural networks.
The first experiment used a nonlinear model with multiple asymptotically equilibrium points.
The result of this experiment showed that our proposed method can robustly estimate a system with multiple stable points.
In the next experiment, we applied our method to a biological system using a simulator of the glucose-insulin system.
It was confirmed that the proposed method can successfully learn a proper system that works well under unexpectedly large inputs due to the $\mathcal{L}_2$ stability.
There is a limitation that our method cannot apply the system without the $\mathcal{L}_2$ stability.
Future work will expand the $\mathcal{L}_2$ stability-based method to a more generalized learning method by dissipativity and apply our approach of this study to stochastic systems.

\section*{Acknowledgements}
This research was supported by JST Moonshot R\&D Grant Number JPMJMS2021 and JPMJMS2024. This work was also supported by JSPS KAKENHI Grant No.21H04905 and CREST Grant Number JPMJCR22D3, Japan. This paper was also based on a part of results obtained from a project commissioned by the New Energy and Industrial Technology Development Organization (NEDO).

{
\small
\bibliographystyle{unsrt}
\bibliography{references}
}



\appendix
\newpage

\section{Proof of Theorem~1 and Corollary~2}\label{APP1}
To prove Theorem~\ref{Thm:optimization} and Corollary~\ref{Coro:optimization_fix_h}, we use the following particular solution of QCQP problems.
\begin{lemma}\label{lem:QCQP}
Suppose that $A \in \R^{n\times n}$ is a positive definite matrix, a vector $b\in \R^{n}$, a scalar $c$, and positive constants $k_{\rm x}$, $k_{\rm y}$.
The solution of this problem
\begin{align*}
    \text{\bf minimize} \quad &k_{\rm x}x^\T A x  +  k_{\rm y}|y|\\ 
    \text{\bf subject to}\quad   & y \geq x^\T A x - 2b^\T x + c.
\end{align*}
is given by
\begin{align*}
x^* &= \left(1 - \sqrt{\mathrm{C}(1 - \tilde{c};\tilde{k}_{\rm x}^2,1)} \right)A^{-1}b,\\
y^* &= \mathrm{R}\left(c -  (1 - \tilde{k}_x^2)b^\T A^{-1}b\right),
\end{align*}
where
\begin{align*}
    \tilde{c} = \frac{c}{b^\T A^{-1}b},\quad \tilde{k}_{\rm x} = \frac{k_{\rm x}}{k_{\rm x} + k_{\rm y}}.
\end{align*}
\end{lemma}
\begin{proof}
See Appendix~\ref{APP2}.
\end{proof}
To keep the notation short in the following proofs, we define 
\begin{align*}
\beta \triangleq \nabla V.
\end{align*}

{\bf Proof of  Theorem~2:}
For the Hamilton-Jacobi function $\mathrm{HJ}(\fm,\Gm,\hm)$ of the modified dynamics $(\fm,\Gm,\hm)$, the map $\fm$ and each row of the map $\Gm$ don't depend on the orthogonal vector of $\beta$.
Hence, $\fm$ and $\Gm$ are written as
\begin{align*}
    \fm = \fn - \beta p,\quad
    \Gm = \Gn - \beta q^\T,
\end{align*}
where $p$ is scalar and $q\in \R^m$.
$\hm$ depends only on the norm in $\mathrm{HJ}(\fm,\Gm,\hm)$, i. e.,
\begin{align*}
\hm =  (1 - r)hn,
\end{align*}
where $r$ is scalar.

The optimal function~(\ref{Eq:thm_optimization_minimize}) is  rewritten as the $p$, $q$, and $r$, i.e.,
\begin{align*}
    &\frac{k_1}{\|\beta\|}\|\fm-\fn\| + \frac{k_2}{2 \gamma^2} \|\Gm-\Gn\|^2 + \frac{k_2}{2 \|\beta\|^2}\|\hm - \hn\| \\
    &=k_1 |p| + k_2\frac{\|\beta\|^2}{2 \gamma^2} \|q\|^2  + k_2 \frac{\|\hn\|}{2 \|\beta\|^2} r^2.
\end{align*}
Also, the condition of constraint~(\ref{Eq:thm_optimization_subject}) is rewritten as
\begin{align*}
&\mathrm{HJ}_{\beta} (\fm,\Gm,\hm)\\
&=\beta^\T \fm + \frac{1}{2\gamma^2}\| \Gm^\T \beta \|^2 + \frac{1}{2}\| \hm\|^2\\
&=\beta^\T (\fn  - \beta p) + \frac{1}{2\gamma^2}\| (\Gn -\beta q^\T) ^\T \beta \|^2 +\frac{1}{2}\Big\|(1-r)\hn\|^2 \\
&= -\|\beta\|^2 p + \frac{1}{2\gamma^2}\|\beta \|^4 \|q\|^2  + \frac{1}{2}\|\hat{h}\|^2 r^2 - \frac{1}{\gamma^2} \| \beta\|^2 \beta^\T \Gn q  - \|\hn\|^2 r + \mathrm{HJ}_{\beta} (\fn,\Gn,\hn)\\
&\leq 0.
\end{align*}

If $x \triangleq [q^\T, r]^\T$, $y \triangleq p$, and
\begin{align*}
A \triangleq 
\begin{bmatrix}
    \frac{\|\beta\|^2}{2\gamma^2} I_m & 0\\
    0 & \frac{\|\hn \|^2}{2\|\beta\|^2}
\end{bmatrix}, \quad
b \triangleq
\begin{bmatrix}
    \frac{1}{2\gamma^2}\Gn^\T \beta \\
    \frac{\|\hn\|^2}{2\|\beta\|^2}
\end{bmatrix}\quad c \triangleq \frac{1}{\|\beta\|^2}\Big(\mathrm{HJ}_\beta(\fn,\Gn,\hn) \Big), 
\end{align*}
then this optimal problem~(\ref{Eq:thm_optimization}) of this theorem becomes a problem used in Lemma~\ref{lem:QCQP}.
The optimal point of $p$, $q$ and $r$ is given by
\begin{align*}
    p^* &=  \frac{1}{\|\beta\|^2}\mathrm{R}\left(V_{\hat{f}} + \hat{k}^2 V_{\Gn, \hn}\right),\\
    q^* &= \frac{1}{\|\beta\|^2}\left(1 - \sqrt{\mathrm{C}\Big( -\tfrac{V_{\fn}}{V_{\Gn,\hn}};\tilde{k}^2,1\Big)} \right)\Gn^\T\beta,\\
    r^* &= 1 - \sqrt{\mathrm{C}\Big( - \tfrac{V_{\fn}}{V_{\Gn,\hn}};\tilde{k}^2,1\Big)}.
\end{align*}
Therefore, Theorem~\ref{Thm:optimization} is derived.

{\bf Proof of  Theorem~2:}

Optimal maps $\fm$ and $\Gm$ are also written as
\begin{align*}
    \fm =  \fn - \beta p,\quad
    \Gm =  \Gn - \beta q^\T.
\end{align*}
Hence, the objective function~(\ref{Eq:coro2_optimization_minimize}) is rewritten as $p$ and $q$ function, i.e.,
\begin{align*}
    &\frac{k_1}{\|\beta\|}\|\fm-\fn\| + \frac{k_2}{2 \gamma^2} \|\Gm-\Gn\|^2 = k_1 |p| + k_2\frac{\|\beta\|^2}{2 \gamma^2} \|q\|^2.
\end{align*}
Also the condition of constraint~(\ref{Eq:coro2_optimization_subject}) is written
\begin{align*}
&\mathrm{HJ}_{\beta} (\fm,\Gm,\hn)\\
&=\beta^\T \fm + \frac{1}{2\gamma^2}\| \Gm^\T \beta \|^2 + \frac{1}{2}\| \hm\|^2\\
&=\beta^\T (\fn  - \beta p) + \frac{1}{2\gamma^2}\| (\Gn -\beta q^\T) ^\T \beta \|^2 +\frac{1}{2}\|\hn\|^2\\
&= -\|\beta\|^2 p + \frac{1}{2\gamma^2}\|\beta \|^4 \|q\|^2  - \frac{1}{\gamma^2} \| \beta\|^2 \beta^\T \Gn q + \mathrm{HJ}_{\beta} (\fn,\Gn,\hn)\\
&\leq 0.
\end{align*}
If $x \triangleq q$, $y \triangleq p$, and
\begin{align*}
A \triangleq \frac{\|\beta\|^2}{2\gamma^2} I_m,\quad 
b \triangleq \frac{1}{2\gamma^2}\Gm^\T \beta,
\quad c \triangleq \frac{1}{\|\beta\|^2}\Big(\mathrm{HJ}_\beta(\fn,\Gn,\hn) \Big),
\end{align*}
then this optimal problem~(\ref{Eq:coro2_optimization}) becomes the problem used in Lemma~\ref{lem:QCQP}.
The optimal points of $p$ and $q$ are given by
\begin{align*}
    p^* &=  \frac{1}{\|\beta\|^2}\mathrm{R}\left(V_{\fn, \hn}  +  \hat{k}_2^2(V_{\Gn})\right),\\
    q^* &= \frac{1}{\|\beta\|^2}\left(1 - \sqrt{\mathrm{C}\Big( -\tfrac{V_{\fn,\hn}}{V_{\Gn}};\tilde{k}^2,1\Big)} \right)\Gn^\T\beta,
\end{align*}
Therefore, the solution of problem~(\ref{Eq:coro2_optimization}) becomes Eqs.~(\ref{Eq:coro_optimization_solution}).

\section{Proof of the particular solution of QCQP}\label{APP2}
This section presents the proof of the following QCQP problem.
\begin{subequations}
\begin{align}
\text{\bf minimize} \quad &k_{\rm x}x^\T A x  +  k_{\rm y}|y| \label{Eq:QCQP_objective}\\ 
\text{\bf subject to}\quad   & y \geq x^\T A x - 2b^\T x + c, \label{Eq:QCQP_constraint}
\end{align}\label{Eq:QCQP}
\end{subequations}
where $A$ is a positive definite matrix.
We classify this problem according to the parameter $A, b$ and $c$.

First, the solution of this optimal problem is switched depending on the positive or negative value of $c$.
\begin{lemma}\label{lem:QCQP_lemma1}
If $c\leq 0$ then the solution of Eq.~(\ref{Eq:QCQP}) is $(x,y)=(0,0)$.
If $c > 0$, the solution $x^*$ of Eq.~(\ref{Eq:QCQP}) equals the solution of the following problem:
\begin{align}
\text{\bf minimize} \quad &k_x x^\T A x  +  k_y|x^\T A x - 2b^\T x + c|.\label{Eq:QCQP_lemma1}
\end{align}
Furthermore, the solution $y^*$  is given by
\begin{align}
    y^* = k_x x^{*\T} A x^* - 2b^\T x^* + c. \label{Eq:QCQP_lemma1y}
\end{align}
\end{lemma}

\begin{proof}
The objective function~(\ref{Eq:QCQP_objective}) is strictly convex and the minimum point is $(x,y)=(0,0)$. 
If $c\leq 0$, the origin $(x,y)=(0,0)$ satisfies the condition of constraint~(\ref{Eq:QCQP_constraint}).
Therefore, the solution of Eq.~(\ref{Eq:QCQP}) is $(x,y)=(0,0)$.

If $c>0$, $(x,y)=(0,0)$ does not satisfies the constraint condition~(\ref{Eq:QCQP_constraint}), and the optimal solution belong the boundary of the region satisfying~(\ref{Eq:QCQP_constraint}).
Therefore, the solution point $(x^*,y^*)$ satisfies Eq.~(\ref{Eq:QCQP_lemma1y}).
\end{proof}

Also, the solution of the new optimal problem~(\ref{Eq:QCQP_lemma1}) is switched by the ratio between $c$ and $b^\T A^{-1}b$.

\begin{lemma}\label{lem:QCQP_lemma2}
If $\tilde{c} > 1 - \tilde{k}_x^2$ then the solution $x^*$ of the optimal problem~(\ref{Eq:QCQP_lemma1}) is given by
\begin{align}
    x^* = (1 - \tilde{k}_x)A^{-1}b,\label{Eq:lemma2x}
\end{align}
where
\begin{align*}
    \tilde{c} = \frac{c}{b^\T A^{-1}b},\quad \tilde{k}_x = \frac{k_x}{k_x + k_y}.
\end{align*}

If $\tilde{c}\leq 1 - \tilde{k}_x^2$, the solution $x^*$ equals the solution of the following QCQP problem, such that
\begin{subequations}
\begin{align}
    \text{\bf minimize} \quad & x^\T A x \label{Eq:QCQP_lemma2_objective}\\
    \text{\bf subject to}\quad   &x^\T A x - 2b^\T x + c=0.\label{Eq:QCQP_lemma2_constract}
\end{align}\label{Eq:QCQP_lemma2}
\end{subequations}
\end{lemma}

\begin{proof}
The optimal problem~(\ref{Eq:QCQP_lemma1}) is split on the sign of $x^{*\T} A x^* - 2b^\T x^* + c$.

{\bf Case}~$x^{*\T} A x^* - 2b^\T x^* + c < 0$:

The optimal problem of this case is written as
\begin{align*}
\text{\bf minimize} \quad &(k_x - k_y ) x^{\T} A x + 2k_y b^\T x - k_yc.
\end{align*}
If $k_x - k_y \leq 0$, there is no optimal point.
Otherwise $k_x - k_y > 0$, the optimal point is written as
\begin{align*}
    x^* =  -\frac{k_y}{k_x - k_y}A^{-1}b.
\end{align*}
The optimal point does not satisfies the condition $x^{*\T} A x^* - 2b^\T x^* + c < 0$, because 
\begin{align*}
    &x^{*\T} A x^* - 2b^\T x^* + c\\
    &=\frac{k_y^2}{(k_x - k_y)^2} b^\T A^{-1}b + \frac{2k_y}{k_x - k_y} b^\T A^{-1}b + c > 0,
\end{align*}
where $A^{-1}$ is a positive definite matrix and $c>0$.

{\bf Case}~$x^{*\T} A x^* - 2b^\T x^* + c > 0$:

The problem~(\ref{Eq:QCQP_lemma1}) is written as 
\begin{align*}
\text{\bf minimize} \quad &(k_x + k_y ) x^\T A x - 2k_y b^\T x + k_yc.
\end{align*}
Hence, this optimal point is written as 
\begin{align*}
x^*&= \frac{k_y}{k_x + k_y}A^{-1}b\\
&= (1 - \tilde{k}_x)A^{-1}b,
\end{align*}
where $\sqrt{A}$ is a square root of the positive definite matrix $A$.
The condition $x^{*\T} A x^* - 2b^\T x^* + c > 0$ is rewritten by the previous optimal solution, such that
\begin{align*}
    &x^{*\T} A x^* - 2b^\T x^* + c\\
    &= (1 - \tilde{k}_x)^2  b^\T A^{-1}b - 2 (1 - \tilde{k}_x)  b^\T A^{-1}b + c\\
    &=  - \left(1 - \tilde{k}_x^2\right) b^\T A^{-1}b + c > 0\\
    &\Leftrightarrow \tilde{c} > 1 - \tilde{k}_x^2.
\end{align*}
{\bf Case}~$x^{\T} A x^* - 2b^\T x^* + c = 0$:
The problem~(\ref{Eq:QCQP_lemma1}) is written as the QCQP problem~(\ref{Eq:QCQP_lemma2}).
\end{proof}

The solution of the simple QCQP problem~(\ref{Eq:QCQP_lemma2}) is easily derived using the method of Lagrange multiplier.

\begin{lemma}\label{lem:QCQP_lemma3}
The solution of the simple QCQP problem~(\ref{Eq:QCQP_lemma2}) is given by
\begin{align*}
x^* = \left( 1 - \sqrt{1 - \tilde{c}}\right)A^{-1} b.
\end{align*}
\end{lemma}
\begin{proof}
Supposing a Lagrange multiplier $\lambda >0$, the Lagrange function is written as
\begin{align*}
L(x,\lambda) =  x^{\T} A x  + \lambda (x^\T A x - 2 b^\T x + c).
\end{align*}
The KKT condition is given by
\begin{align}
\frac{\partial L(x^*,\lambda)}{\partial x} &=  2(1 + \lambda )A x^* - 2\lambda  b =0. \label{Eq:Lagrange_1}\\
\frac{\partial L(x^*,\lambda)}{\partial \lambda} &=  x^{*\T} A x^* - 2 b^\T x^* + c =0, \label{Eq:Lagrange_2}
\end{align}
Eq.~(\ref{Eq:Lagrange_1}) is written as
\begin{align*}
    x^* &=  \frac{\lambda}{1+\lambda} A^{-1}b,
\end{align*}
and the Eq.~(\ref{Eq:Lagrange_2}) is given by
\begin{align*}
    &x^{*\T} A x^* - 2 b^\T x^* + c =0\\
    \Leftrightarrow&\frac{\lambda^2}{(1+\lambda)^2}b^\T A^{-1}b - \frac{2\lambda}{1+\lambda}b^\T A^{-1}b + c = 0\\
    \Leftrightarrow &  - \frac{\lambda^2 + 2\lambda}{(1+\lambda)^2}b^\T A^{-1}b + c = 0\\
    \Leftrightarrow &   - (\lambda^2 + 2\lambda) b^\T A^{-1}b + (1+\lambda)^2c = 0\\
    \Leftrightarrow &   (c -  b^\T A^{-1}b )\lambda^2 +  2(c -b^\T A^{-1}b )\lambda + c = 0\\
    \Leftrightarrow &   \lambda^2 +  2\lambda + \frac{\tilde{c}}{(\tilde{c} -  1 )} = 0\\
    \Leftrightarrow &   \lambda =  -1 \pm \sqrt{1 - \frac{\tilde{c}}{\tilde{c} -  1}}\\
    \Leftrightarrow &   \lambda =  -1 \pm \sqrt{\frac{1}{1-\tilde{c}}}.
\end{align*}
As $\lambda >0$ and $\tilde{c}>0$, the Lagrange multiplier is written as
\begin{align*}
\lambda =  - 1 + \sqrt{\frac{1}{1-\tilde{c}}}.
\end{align*}
Therefore, the optimal point $x^*$ is given by
\begin{align*}
    x^* &= \frac{\lambda}{1 + \lambda} A^{-1}b\\
    &=\frac{-1 + \sqrt{\frac{1}{1-\tilde{c}}}}{\sqrt{\frac{1}{1-\tilde{c}}}} A^{-1}b\\
    &= (1 - \sqrt{1-\tilde{c}})A^{-1}b.
\end{align*}
\end{proof}
Finally, we summarize three Lemmas~\ref{lem:QCQP_lemma1}-\ref{lem:QCQP_lemma3} and solve the QCQP problem~(\ref{Eq:QCQP}).
\setcounter{lemma}{0}
\begin{lemma}
The solution of the problem~(\ref{Eq:QCQP}) is given by
\begin{subequations}
\begin{align}
x^* &= \left(1 - \sqrt{\mathrm{C}\left(1 - \tilde{c};\tilde{k}_x^2,1 \right)} \right)A^{-1}b,\\
y^* &= \mathrm{R}\left(c -  (1 - \tilde{k}_x^2)b^\T A^{-1}b\right).
\end{align}\label{Eq:QCQP_prop}
\end{subequations}
\end{lemma}
\begin{proof}
Lemmas~\ref{lem:QCQP_lemma1}-\ref{lem:QCQP_lemma3} split the solution~$x^*$ as three cases:
\begin{align*}
x^* &:= 
\begin{cases}
(1 - \tilde{k}_x)A^{-1}b &   1 - \tilde{k}_x^2 <\tilde{c}\\
\left( 1 - \sqrt{1 - \tilde{c}}\right)A^{-1} b &  0< \tilde{c} \leq 1 - \tilde{k}_x^2\\
0 &  c\leq 0
\end{cases}\\
&:= 
\begin{cases}
\left(1 - \sqrt{\tilde{k}_x^2} \right)A^{-1}b &      1 -\tilde{c} < \tilde{k}_x^2\\
\left( 1 - \sqrt{1 - \tilde{c}}\right)A^{-1} b &  \tilde{k}_x^2 \leq   1- \tilde{c} < 1 \\
\left(1 - \sqrt{1}\right)A^{-1} b &  1 \leq 1- \tilde{c}
\end{cases}\\
&= \left(1 - \sqrt{\mathrm{C}\left(1 - \tilde{c};\tilde{k}_x^2,1 \right)} \right)A^{-1}b.
\end{align*}
Furthermore, the solution of $y^*$ is written as,
\begin{align*}
y^* &:= 
\begin{cases}
x^{*\T} A x^* - 2b^\T x^* + c &   1 - \tilde{k}_x^2 <\tilde{c}\\
x^{*\T} A x^* - 2b^\T x^* + c &  0< \tilde{c} \leq 1 - \tilde{k}_x^2\\
0 &  c\leq 0
\end{cases}\\
&=
\begin{cases}
c -  (1 - \tilde{k}_x^2)b^\T A^{-1}b &   1 - \tilde{k}_x^2 <\tilde{c}\\
0 &  0< \tilde{c} \leq 1 - \tilde{k}_x^2\\
0 &  \tilde{c}\leq 0
\end{cases}\\
&=  \mathrm{R}(c -  (1 - \tilde{k}_x^2)b^\T A^{-1}b).
\end{align*}
Therefore, the solution become Eq.~(\ref{Eq:QCQP_prop}).
\end{proof}

\section{Overall schematic of the learning process}\label{APP4}

Algorithm \ref{alg:a1} shows the overall schematic of the learning process.
The first line defines the modified dynamics $(\fm,\Gm,\hm)$ from the nominal dynamics $(\fn,\Gn,\hn)$, defined by the neural network, where $\phi$ is a set of parameters of the nominal dynamics.
The 2-7 line represents a training loop, where the gradient-based optimization methods can be used by using the forward and backward calculation.
Note that an ODE solver is used for forward calculation,
and Algorithm \ref{alg:a2} shows the forward calculation when the Euler method is used.
For simplicity, mini-batch computation omitted in this schematic.

 \begin{algorithm}
 \caption{Training process}
 \label{alg:a1}
 \begin{algorithmic}[1]
 \renewcommand{\algorithmicrequire}{\textbf{Input:}}
 \renewcommand{\algorithmicensure}{\textbf{Output:}}
 \REQUIRE $x_0$: initial state, $u$: input signal,$y$: output signal,   $(\fn,\Gn,\hn)$: nominal dynamics, $V$: a designed function
  \STATE define modified functions $(\fm,\Gm,\hm)$ from $(\fn,\Gn,\hn) $ and $V$
  \FOR {$1$ to $\#$iterations}
    \STATE $\hat{y} \leftarrow$ ODE with $(\fm,\Gm,\hm)$ from $x_0, u$ ({\bf Algorithm~2})
    \STATE forward computation of ${\rm Loss}$ function (\ref{Eq:HJ_Loss}) from $y$
    \STATE $\nabla_\phi {\rm Loss} \leftarrow $ backward computation with ${\rm Loss}$
    \STATE $\phi \leftarrow$ Optimizer($\phi$, $\nabla_\phi {\rm Loss}$)
  \ENDFOR
 \end{algorithmic} 
 \end{algorithm}

\begin{algorithm}
 \caption{Forward computation for dynamics Eq. (1)}
 \label{alg:a2}
 \begin{algorithmic}[1]
 \renewcommand{\algorithmicrequire}{\textbf{Input:}}
 \renewcommand{\algorithmicensure}{\textbf{Output:}}
 \REQUIRE $x_0$: initial state, $u$: input signal, $(\fm,\Gm,\hm)$: dynamics 
 \ENSURE  $\hat{y}$: output signal \\
  \FOR {t $\leftarrow 0$ to $T$}
    \STATE $x_{t+1} \leftarrow x_t + \Delta t (\fm(x_t)+\Gm(x_t)u_t)$
    \STATE $\hat{y}_{t} \leftarrow \hm(x_{t})$
  \ENDFOR
 \RETURN $\hat{y}$ 
 \end{algorithmic} 
\end{algorithm}

\section{Neural network architecture and hyper parameters}

\begin{table}[t]
    \centering
    \caption{The search space of Bayesian optimization}
    \begin{tabular}{|l|c|c|} \hline
parameter name & range    & type\\ \hline \hline
learning rate & $10^{-5}$  -- $10^{-3}$ & log scale \\ \hline
weight decay  & $10^{-10}$  -- $10^{-6}$ & log scale \\ \hline
batch size    & $ 10   $ -- $100 $ & integer \\ \hline
optimizer     & $\{$ AdamW, Adam, RMSProp$\}$  & categorical\\ \hline
$\#$layer for $\fn$ & $0$ -- $3$ & integer \\ \hline
$\#$layer for $\Gn$ & $0$ -- $3$ & integer\\ \hline
$\#$layer for $\hn$ & $0$ -- $3$ & integer\\ \hline
$\#$dim. for a hidden layer of $\fn$ & $8 - 32$ & integer\\ \hline
$\#$dim. for a hidden layer of $\Gn$ & $8 - 32$ & integer \\ \hline
$\#$dim. for a hidden layer of $\hn$ & $8 - 32$ & integer \\ \hline
$\epsilon$  & $0$ -- $1.0 $ & log scale \\ \hline
Initial scale parameter for $\fn$  & $10^{-5}$ -- $0.1 $ & log scale \\ \hline
Stop gradient for projection  &  {true, false} & boolean\\ \hline
    \end{tabular}
    \label{tab:bo}
\end{table}

This section details how to determine the neural network architecture.
The architecture and hyper parameters of the neural networks were basically determined by Bayesian optimizing using the validation dataset.

Table \ref{tab:bo} shows the search space of Bayesian optimization.
The first three parameters: learning rate, weight decay, and batch size are parameters for training the neural networks.
Also, an optimizer is selected from AdamW, Adam, and RMSProp.
The structure of neural network is determined from the number of intermediate layers and  dimensions for each layer.
One layer in our setting consists of a fully connected layer with a ReLU activation.
The last three rows represent parameters related to our proposed methods.
$\epsilon$ is a parameter of the loss function $L_{{\rm HJ}}$.
Initial scale parameter is multiplied with the output of $\fn$ to prevent the value of $\fn(x)$ from becoming large in the initial stages of learning.
When $\fn(x)$, which determine the behavior of the internal system, outputs a large value, it diverges due to time evolution, and the learning of the entire system may not progress. 
Therefore, it is empirically preferable to start with a small value for $\fn(x)$ at the initial stage of learning.
When the flag of `stop gradient for projection' is false,
backward computation related to the second term of modification of $\fm$ and $\Gm$ is disabled.
Note that modification related to $\fm$ and $\Gm$ consists of two terms (see Theorem 1).
Setting this parameter to false resulted in better performance in our all experiments.

We ran 300 trials using the Bayesian optimization for the bistable model benchmark and the glucose insulin benchmark with the above settings, and the hyper parameters obtained are shown in Table \ref{tab:hparam}, where the number of dimensions for each hidden layer is shown in tuple from the order closest to the input layer.

Hyperparameters of comparative methods were determined by grid search using the validation dataset.
ARX and PWARX have an order parameter $n$ of the autoregressive model, and this parameter is searched in the range of $1-5$.
The number of iterations was set to 10000 so that the optimization of PWARX converges sufficiently.
MOESP and ORT have an internal dimension $n$ ($1 \leq  n \leq 20$)
and the number of subsequences used for estimation $k$  ($2n < k \leq 20$).

\begin{table}[t]
    \centering
    \caption{Selected parameters for each benchmark}
    \begin{tabular}{|l|c|c|} \hline
parameter name & bistable    & glucose insulin\\ \hline \hline
learning rate & $3.01\times 10^{-4}$ & $3.28\times 10^{-4}$ \\ \hline
weight decay  & $4.76\times 10^{-9}$ & $2.28\times 10^{-9}$ \\ \hline
batch size    & $100 $ & $100$ \\ \hline
optimizer     &  RMSProp  & RMSProp\\ \hline
$\#$layer for $\fn$ & $3$ & $1$ \\ \hline
$\#$layer for $\Gn$ & $1$ & $2$ \\ \hline
$\#$layer for $\hn$ & $3$ & $2$ \\ \hline
$\#$dim. for a hidden layer of $\fn$ & (17,10,22) &  (8)\\ \hline
$\#$dim. for a hidden layer of $\Gn$ & (34)       &  (27,29) \\ \hline
$\#$dim. for a hidden layer of $\hn$ & (10,62,58) &  (35,18) \\ \hline
$\epsilon$  & $0.63$ & $0.75$ \\ \hline
Initial scale parameter for $\fn$  & $9.64 \times 10^{-2}$ & $8.94 \times 10^{-2}$ \\ \hline
Stop gradient for projection  &  false & false \\ \hline
    \end{tabular}
    \label{tab:hparam}
\end{table}

\newpage
\section{Glucose-Insulin system}

Glucose concentration in the blood is modeled as a time-delay system regulated by insulin concentration (See Fig.~\ref{fig:GIR} (A)) \cite{DEGAETANO200041}.
Suppose that $G$ $I$, and $X$ are the glucose, insulin, and accumulated glucose plasma concentration  ([mg/100ml],[$\mu$UI/ml],and [min mg/100ml], respectively) and $u$ is the amount of ingested glucose per minute [min${}^{-1}$ mg/100ml].
The dynamics of each concentration is given by 
\begin{align*}
\dot{G}(t) &= - k_1 G(t) - k_2 G(t) I(t) + g_0 + u(t),\\
\dot{I}(t) &= - k_3 I(t) + \frac{k_4}{\tau} X(t),\\
\dot{X}(t) &= G(t) - G(t-\tau),\\
y(t) &= [G(t),I(t)]^{\mathrm{T}},
\end{align*}
where $k_1$ is a spontaneous glucose disappearance rate, $k_2$ is an insulin-dependent glucose disappearance rate, $g_0$ is a constant increase in plasma glucose concentration, $k_3$ is an insulin disappearance rate, $k_4$ is an insulin release rate per the average glucose concentration within the last $\tau$ minute.

This system has a unique asymptotically stable equilibrium point $(G, I, X) \equiv (G^*, I^*, X^*)$ on the nonlinear plain such that
\begin{align*}
G^* &= \frac{-k_1k_3 + \sqrt{(k_1k_3)^2 + 4k_2k_3k_4g_0}}{2k_2k_4},\\
I^* &= \frac{-k_1k_3 + \sqrt{(k_1k_3)^2 + 4k_2k_3k_4g_0}}{2k_2k_3},\\
X^* &= \frac{-k_1k_3 + \sqrt{(k_1k_3)^2 + 4k_2k_3k_4g_0}}{2k_2k_4} \tau.
\end{align*}
Here we set the initial state of this model as $G(0) = G^*,I(0) = I^*, X(0) = X^*$ and set the model parameters as shown in the following Table~\ref{tab:GIS}.
Furthermore, we adopt the output of the previous oral glucose absorption system~\cite{DallaMan2007} as $u$ (See Fig.~\ref{fig:GIR} (B)), and the glucose absorption amount $u$ is normalized based on human blood volume per body weight (0.80[100 ml/kg]). 

\begin{table}[t]
    \centering
    \caption{Model parameters.}
    \begin{tabular}{|c|c|c|}
    \hline
    {\bf Parameter} &{\bf Value}&{\bf Unit} \\\hline\hline
    $k_{1}$ &  $3.35\times 10^{-2}$ & $\frac{1}{\text{min}}$ \\ \hline
    $k_{2}$ &  $5.22\times 10^{-5}$ & $\frac{1}{\text{min}(\mu \text{UI/ml})}$\\ \hline
    $k_{3}$ &  $1.055$ & $\frac{1}{\text{min}}$\\ \hline
    $k_{4}$ &  $0.293$ & $\frac{ (\mu \text{UI/ml})}{\text{min} \text{(mg/100ml)}}$  \\ \hline
    $g_0$ &  3.13 & $\frac{\text{(mg/100ml)}}{\text{min}}$ \\ \hline
    $\tau$ & $6$ & min \\ \hline
    \end{tabular}
    \label{tab:GIS}
\end{table}

\begin{figure}[t]
    \centering
    \begin{tabular}{cc}
    \includegraphics[width=0.45\linewidth]{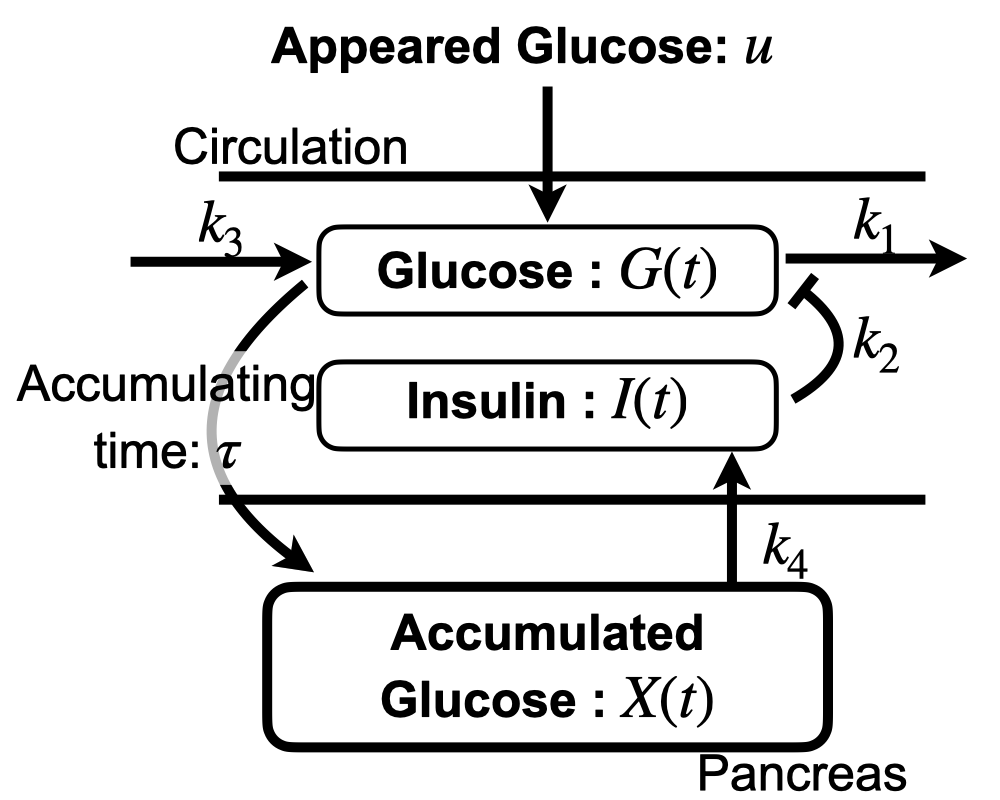}     &  \includegraphics[width=0.45\linewidth]{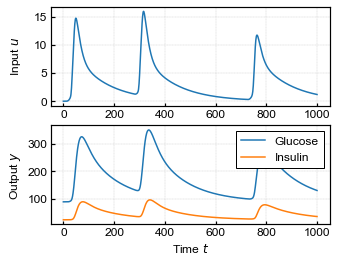}\\
    \text{(A)} & \text{(B)}
    \end{tabular}
    \caption{(A) Overview and  (B) input and output behavior of the glucose insulin system.}
    \label{fig:GIR}
\end{figure}

\end{document}